\documentclass[reqno,11pt]{amsart}
\usepackage{amsmath,amssymb,amsfonts,amsthm, amscd,indentfirst}
\usepackage{amsmath,latexsym,amssymb,amsmath,%awheremscd,
	amscd,amsthm,amsxtra}
\usepackage{color}
\usepackage{pdfpages}
\usepackage{graphics}
\usepackage{enumitem}

\usepackage[usenames,dvipsnames]{xcolor}

\usepackage{varioref}
\usepackage{hyperref}
\usepackage[noabbrev,capitalise,nameinlink]{cleveref}

\hypersetup{colorlinks={true},linkcolor={blue},citecolor=blue}
\usepackage[english]{babel}% Recommended
\usepackage{csquotes}% Recommended

\usepackage[backend=biber,
style=alphabetic,datamodel=ext-eprint,maxbibnames=99]{biblatex}
\DeclareSourcemap{
	\maps[datatype=bibtex]{
		\map{
			\step[fieldsource=pmid, fieldtarget=pubmed]
		}
	}
}
\makeatletter
\DeclareFieldFormat{arxiv}{%
	arXiv\addcolon\space
	\ifhyperref
	{\href{http://arxiv.org/\abx@arxivpath/#1}{%
			\nolinkurl{#1}%
			\iffieldundef{arxivclass}
			{}
			{\addspace\texttt{\mkbibbrackets{\thefield{arxivclass}}}}}}
	{\nolinkurl{#1}
		\iffieldundef{arxivclass}
		{}
		{\addspace\texttt{\mkbibbrackets{\thefield{arxivclass}}}}}}
\makeatother
\DeclareFieldFormat{pmcid}{%
	PMCID\addcolon\space
	\ifhyperref
	{\href{http://www.ncbi.nlm.nih.gov/pmc/articles/#1}{\nolinkurl{#1}}}
	{\nolinkurl{#1}}}
\DeclareFieldFormat{mr}{%
	MR\addcolon\space
	\ifhyperref
	{\href{http://www.ams.org/mathscinet-getitem?mr=MR#1}{\nolinkurl{#1}}}
	{\nolinkurl{#1}}}
\DeclareFieldFormat{zbl}{%
	Zbl\addcolon\space
	\ifhyperref
	{\href{http://zbmath.org/?q=an:#1}{\nolinkurl{#1}}}
	{\nolinkurl{#1}}}
\DeclareFieldAlias{jstor}{eprint:jstor}
\DeclareFieldAlias{hdl}{eprint:hdl}
\DeclareFieldAlias{pubmed}{eprint:pubmed}
\DeclareFieldAlias{googlebooks}{eprint:googlebooks}

\renewbibmacro*{eprint}{%
	\printfield{arxiv}%
	\newunit\newblock
	\printfield{jstor}%
	\newunit\newblock
	\printfield{mr}%
	\newunit\newblock
	\printfield{zbl}%
	\newunit\newblock
	\printfield{hdl}%
	\newunit\newblock
	\printfield{pubmed}%
	\newunit\newblock
	\printfield{pmcid}%
	\newunit\newblock
	\printfield{googlebooks}%
	\newunit\newblock
	\iffieldundef{eprinttype}
	{\printfield{eprint}}
	{\printfield[eprint:\strfield{eprinttype}]{eprint}}}

\usepackage{graphicx}
\usepackage{caption}
\usepackage{epsfig,here}
\usepackage{subfigure,here}

\newtheorem{theorem}{Theorem}[section]%[section]
\newtheorem{lemma}[theorem]{Lemma}%[section]
\newtheorem{proposition}[theorem]{Proposition}%[section]
\newtheorem{corollary}[theorem]{Corollary}%[section]
\newtheorem{remark}[theorem]{Remark}

\def\om{\omega}
\def\O{\Omega}
\def\Om{\Omega}
\def\p{\partial}

\def\De{\Delta}

\def\S{{\Sigma}}
\def\<{\langle}
\def\>{\rangle}

\def\na{\nabla}
\def\Ga{\Gamma}

\def\rr{\mathbf{R}}
\newcommand{\mcH}{\mathcal{H}}

\newcommand{\mcW}{\mathcal{W}}

\newcommand{\mfR}{\mathbb{R}}
\newcommand{\mfS}{\mathbb{S}}

\newcommand{\ra}{\rightarrow}

\newcommand{\rd}{{\rm d}}

\numberwithin{equation} {section}
\begin{document}
	
	\title{Capillary Schwarz symmetrization in the half-space}
	
	\author{Zheng Lu}
	\address{School of Mathematical Sciences\\
		Xiamen University\\
		361005, Xiamen, P.R. China
  \newline\indent\&
\newline\indent
Mathematisches Institut\\ Albert-Ludwigs-Universität Freiburg\\ Freiburg im Breisgau, 79104, Germany}
  %\address{Mathematisches Institut\\ Albert-Ludwigs-Universität Freiburg\\ Freiburg im Breisgau, 79104, Germany}
	\email{zhenglu@stu.xmu.edu.cn}
	
	\author{Chao Xia}
	\address{School of Mathematical Sciences\\
		Xiamen University\\
		361005, Xiamen, P.R. China}
	\email{chaoxia@xmu.edu.cn}
	
	\author{Xuwen Zhang}
	\address{School of Mathematical Sciences\\
		Xiamen University\\
		361005, Xiamen, P.R. China
  \newline\indent\&
\newline\indent Institut f\"ur Mathematik, Goethe-Universit\"at, %Robert-Mayer-Str. 10,
60325, Frankfurt, Germany}
	\email{xuwenzhang@stu.xmu.edu.cn}
	\thanks{This work is  supported by NSFC (grant no. 12271449).
		}
	
	\begin{abstract}
	In this paper, we introduce a notion of capillary Schwarz symmetrization in the half-space. It can be viewed as the counterpart of the classical Schwarz symmetrization in the framework of capillary problem in the half-space. A key ingredient is a special anisotropic gauge, which enables us to transform the capillary symmetrization to the convex symmetrization introduced in \cite{AFT97}.
		
		\
		
		
	\end{abstract}
	
	\maketitle
	
	\medskip
%---------------------------------------------Section Style---------
\section{Introduction}

Symmetrization is an important technique to prove sharp geometric or functional inequalities. Schwarz symmetrization is a classical one which assigns to a given function, a radially symmetric function whose super or
sub level-sets have the same volume as that of the given function. Important applications include the proof of the Rayleigh-Faber-Krahn inequality on first eigenvalue and the sharp Sobolev inequality, see \cite{PS51, Talenti76a}. 

The classical Schwarz symmetrization is based on the classical isoperimetric inequality. It is in fact a common principle that a symmetrization process is usually accompanied with an isoperimetric-type inequality. Several new kinds of symmetrization has been introduced, for example, Talenti \cite{Talenti81} and Tso \cite{Tso89} introduces the symmetrization with respect to quermassintegrals, based on Alexandrov-Fenchel inequalities for quermassintegrals. 
Alvino-Ferone-Trombetti-Lions \cite{AFT97} introduces the convex symmetrization with respect to convex gauge functions (or anisotropic functions), based on anisotropic isoperimetric inequality. Della Pietra, Gavitone and the second-named author \cite{DPGX21} introduces symmetrization with respect to mixed volumes, based on Alexandrov-Fenchel inequalities for mixed volume. 

Let $\mfR^n_+:=\{x\in \mfR^n: \<x, E_n\>>0\}$ be the upper half-space, where $E_n$ is the $n$-th coordinate unit vector.
The relative isoperimetric inequality, due to De Giorgi, says that for  $\theta\in(0,\pi)$ and a set of finite perimeter $E\subset\mfR^n_+$, it holds that
\begin{align}\label{isop-ineq}
    \frac{P(E; \mfR^n_+)-\cos\theta P(E; \p\mfR^n_+)}{|E|^{\frac{n-1}{n}}}\ge \frac{P(\mathcal{B}; \mfR^n_+)-\cos\theta P(\mathcal{B}; \p\mfR^n_+)}{|\mathcal{B}|^{\frac{n-1}{n}}},
\end{align}
where $\mathcal{B}$ denotes the domains  $B_1(-\cos\theta E_n)\cap \mfR^n_+$, and
equality holds in \eqref{isop-ineq} if and only if $E=B_r(-r\cos\theta E_n)\cap \mfR^n_+$ for some $r>0$. Here $B_r(-r\cos\theta E_n)$ denotes the Euclidean ball of radius $r$ centered at $-r\cos\theta E_n$. Such family of balls shares the common property that their boundaries intersect $\p \mfR^n_+$ at the constant contact angle $\theta$. The functional $P(E; \mfR^n_+)-\cos\theta P(E; \p\mfR^n_+)$ is usually referred to as the free energy functional in capillarity problem, which is natural in the physical model of liquid drops, see for example \cite{Mag12}.

Our purpose of this paper is to introduce a suitable symmetrization, which we shall call capillary Schwarz symmetrization, to be accompanied with the isoperimetric-type inequality \eqref{isop-ineq}.

For a non-positive measurable function $u$ defined on $\mfR^n_+$, We define the capillary Schwarz symmetrization to be
\begin{align*}
    u_\star(x)=\sup\{t\leq0: r_t< \vert x+r_t\cos\theta E_n\vert\},
\end{align*}
where $r_t>0$ is such that $|B_{r_t}(-{r_t}\cos\theta E_n)\cap \mfR^n_+|=|\{u(x)< t\}|$. It is clear by definition that
$|\{u_\star(x)< t\}|=|\{u(x)< t\}|$ and the level-sets for $u_\star(x)$ are the desired model domains $B_r(-r\cos\theta E_n)\cap\mfR^n_+$.

In order to study the property of capillary Schwarz symmetrization, we introduce the following convex gauge 
$F_\theta:\mfR^n\ra\mfR^+$ given by
\begin{align*}
    F_\theta(\xi)=\vert\xi\vert-\cos\theta\left<\xi,E_n\right>.
\end{align*}
We shall call it \textit{capillary gauge}.
One crucial observation for $F_\theta$ is that the Wulff ball of radius $r$ with respect to $F_\theta$,  $\{F_\theta^o(x)<r\}$, is equivalent to  $B_r(-r\cos\theta E_n)$, for any $r>0$.  
This enables us to transform the capillary Schwarz symmetrization to the convex symmetrization (due to Alvino-Ferone-Trombetti-Lions \cite{AFT97}) with respect to $F_\theta$. Compare to \cite{Talenti76, AFT97}, there are two major differences. One is that the special gauge $F_\theta$ is not even, and the other is that we consider the relative version of convex symmetrization. Nevertheless, we are able to show the P\'olya-Szeg\"o principle and the PDE comparison result for such relative convex symmetrization, with an additional non-positive (or non-negative) requirement on functions, following the proof of \cite{AFT97}. More generally, the result holds true in any convex cones where the relative anisotropic isoperimetric inequality holds, see for example \cite{CRS16, DPGV22}.
The corresponding results eventually can be transformed to the capillary Schwarz symmetrization with the help of $F_\theta$.

We remark that the idea of transforming the capillary Schwarz symmetrization to the convex symmetrization is inspired by recent work of De Phillipis-Maggi \cite{PM15}, where they use similar idea to transform regularity of local minimizers in capillarity problems to that in anisotropic problems.
The idea may have future applications in other capillary problems. Here we mention one such application.
In \cite{JWXZ22-A}, Jia et al. proved the following Heintze-Karcher-type inequality for capillary hypersurfaces in $\mfR^n_+$: for a bounded domain $E$ with $\p E\cap \mfR^n_+$ sufficiently smooth and intersecting $\mfR^n_+$ at a contact angle $\theta$, there holds
\begin{align}\label{hk-ineq}
    \int_{\p E\cap \mfR^n_+}\frac{1-\cos\theta\<\nu, E_n\>}{H}\ge \frac{n}{n-1}|E|,
\end{align}
where $H$ is the mean curvature $\p E\cap \mfR^n_+$. As a consequence, Wente's Alexandrov-type theorem for 
capillary constant mean curvature hypersurfaces in the half-space is reproved. We remark that by the gauge $F_\theta$, \eqref{hk-ineq} can be reformulated as 
\begin{align}\label{hk-ineq1}
    \int_{\p E\cap \mfR^n_+}\frac{F_\theta(\nu)}{H_{F_\theta}}\ge \frac{n}{n-1}|E|,
\end{align}
where $H_{F_\theta}$ is the anisotropic mean curvature, which is equal to $H$, thanks to $\nabla^2 F_\theta\mid_{\nu}={\rm Id}$, the identity matrix.
On the other hand, \eqref{hk-ineq1} is a special case of the result in \cite{JWXZ22-B}, where \eqref{hk-ineq} has been generalized to general anisotropic capillary setting. 

The rest of the paper is organized as follows.
In \cref{Sec-2}, we review the anisotropic isoperimetric inequality in convex cones and study the relative convex symmetrization in convex cones.
In \cref{Sec-3}, we introduce the capillary gauge and study its associated properties.
In \cref{Sec-4}, we introduce the capillary Schwarz symmetrization in the half-space and restate the corresponding results in \cref{Sec-2} by using the capillary gauge.

\

\noindent{\bf Acknowledgements.}
We are indebted to Professor Guofang Wang for stimulating discussions on this topic and his constant support.

\

\section{Convex symmetrization in a convex cone}\label{Sec-2}
%-------------
\subsection{Anisotropic isopermetric inequality in a convex cone}\

In this subsection, we review the basic facts on anisotropic perimeter and anisotropic isoperimetric inequality in a convex cone.

Following \cite[(1.6)]{CRS16}, we say that $F:\mfR^n\ra\mfR$ is a \textit{gauge} if $F$ is non-negative, convex, positively one homogeneous, i.e., $F(t\xi)=tF(\xi)$ for all $t>0$, and $F(\xi)>0$ for all $\xi\in\mfS^{n-1}$.
Note that $F$ is not required to be even, which is important for the applications on capillary symmetrization in next section.
We say that $F$ is a \textit{norm} if in addition, $F$ is even, namely, $F(t\xi)=\vert t\vert F(\xi)$ for any $t\neq0$.

Restricting $F$ on $\mfS^{n-1}$, we get $F: \mfS^{n-1}\to \mfR_+$. The Cahn-Hoffman map is given by 
\begin{align*}
    \Phi:\mfS^{n-1}\ra\mfR^n,\quad\Phi(x):=\na F(x),
    %=x-\cos\theta E_n
\end{align*}
here $\na$ denotes the gradient operator in $\mfR^n$.
The image $\Phi(\mfS^{n-1})$ is called the {\it Wulff shape}.

The corresponding \textit{dual gauge } $F^o:\mfR^{n}\ra\mfR$ is defined by
\begin{align*}
    F^o(x)=\sup\left\{\frac{\left<x,z\right>}{F(z)}\Big| z\in\mfS^{n-1}\right\},
\end{align*}
where $\left<\cdot,\cdot\right>$ denotes the standard Euclidean inner product.
The following identities hold true for gauge and its dual gage:
\begin{align}\label{one-homo}
F(\nabla F^o(x))=1, \quad \nabla_\xi F(\nabla_x F^o_\theta(x))=\frac{x}{F^o(x)}.\end{align}
See for example \cite[(2.8)]{DMMN18}, \cite[Lemma 2.2]{CL22}.

Denote $$\mcW=\{x\in\mfR^{n}|F^o(x)<1\}.$$
We call $\mcW$ the \textit{unit Wulff ball centered at the origin}. One can prove that 
$\p \mcW=\Phi(\mfS^{n-1})$, the Wulff shape. More generally, we denote $$\mcW_r(x_0)=r\mcW+x_0$$ and call it the \textit{Wulff ball of radius $r$ centered at $x_0$}. We simply denote $\mcW_r=\mcW_r(0)$.

%the \textit{unit
%  Wulff shape} $\mcW_F$ can be interpreted by $F^o$ as $$\mcW_F=\{x\in\mbR^{n}|F^o(x)=1\}.$$

Let $\S\subset \mfR^n$ be an open convex cone with vertex at the origin, given by
\begin{align*}
    \S=\{tx: x\in\om, t\in(0,+\infty)\}
\end{align*}
for some open domain $\om\subseteq\mfS^{n-1}$.
The corresponding \textit{Wulff sector in $\S$} is $\mcW\cap \S$.
%  In terms of the convex cone $\S$, we denote by $\mfW_F\subset\S$ the sector-like domain such that it is bounded and $\p\mfW_F\cap\S=\mcW_F\cap\S$.
 % We call $\mfW_F$ the \textit{unit Wulff sector} and denote by $\kappa_n$ its volume (namely, $\kappa_n=\vert \mfW_F\vert$). In particular, if $F$ is the Euclidean norm, we denote by $\om_n$ the volume of $B_1(0)\cap\S$.

%Given a norm $F$ and a convex cone in $\mfR^n$, we define 
For a measurable set $E\subset\rr^n$, the \textit{anisotropic perimeter relative to $\S$} is defined by
\begin{align*}
    P_F(E;\S)=\sup\left\{\int_{E\cap\S}{\rm div}\sigma \rd x: \sigma\in C_0^1(\S;\mfR^n), F^o(\sigma)\leq1\right\}.
\end{align*}
One can check by definition that the quantity $P_F(E;\S)$ is finite if and only if the classical relative perimeter
\begin{align*}
    P(E;\S)=\sup\left\{\int_{E\cap\S}{\rm div}\sigma \rd x: \sigma\in C_0^1(\S;\mfR^n), \vert\sigma\vert\leq1\right\}<\infty.
\end{align*}
In particular, for a set of finite perimeter $E\subset\rr^n$, the anisotropic perimeter (anisotropic surface energy) can be characterized by %(see e.g., \cite{CRS16}, \cite[Section 2.3]{CL22})
\begin{align}\label{eq-P_F}
    P_F(E;\S)=\int_{\p^\ast E\cap \S}F(\nu_E)\rd\mcH^{n-1},
\end{align}
where $\p^\ast E$ is the \textit{reduced boundary} of $E$ and $\nu_E$ is the \textit{measure-theoretic outer unit normal} to $E$.
Note that if $E$ is of $C^1$-boundary
in $\S$, then $\nu_E$ agrees with the classical outer unit normal.

A crucial ingredient for our purpose is the anisotropic isoperimetric inequality in a convex cone given by \cite{CRS16, DPGV22,CL22}.

\begin{theorem}[{\cite[Theorem 1.3]{CRS16}}, {\cite[Theorem 4.2]{DPGV22}}, {\cite[Theorem 2.5]{CL22}}]\label{Thm-CRS16-1-3}
    Let $F$ be a gauge in $\mfR^n$ and $\S$ be an open convex cone with vertex at the origin.
    Then for any measurable set $E\subset\mfR^n$ with $|E\cap \S|<\infty$, there holds
    \begin{align}\label{ineq-anisotropic-isoperimetric}
        \frac{P_F(E;\S)}{\vert E\cap\S\vert^{\frac{n-1}{n}}}\geq \frac{P_F(\mcW;\S)}{\vert \mcW\cap\S\vert^{\frac{n-1}{n}}}.
    \end{align}
    Up to rotations, we may write $\S=\mfR^k\times \tilde{\S}$, where $0\le k\le n$ and $\tilde{\S}\subset \mfR^{n-k}$ is an open convex cone containing no lines. Then equality holds in \eqref{ineq-anisotropic-isoperimetric} if and only if $E$ is a Wulff ball of some radius $r$ centered at $x_0\in \mfR^k\times \{0_{\mfR^{n-k}}\}$.
\end{theorem}

\begin{remark}
\normalfont
In \cite{CRS16}, a more general weighted anisotropic isoperimetric inequality in a convex cone has been proved, although without equality characterization. For unweighted case, the equality has been  characterized in \cite{DPGV22}, following the method of \cite{FI13}.
The original statement in \cite[Theorem 4.2]{DPGV22} is stated for norms. Nevertheless, their proof works without change for general gauges, see \cite[Theorem 2.5]{CL22}.
\end{remark}

%--------------
\subsection{Convex symmetrization in a convex cone}\label{subsec-2-2}\ 

Let $u:\S\to (-\infty, 0]$ be a non-positive measurable function, which vanishes at infinity, in the sense that the \textit{distribution function} 
$$\mu(t)=\left\vert\{x\in\S:
u(x)<t\right\vert\}$$
 is finite for all $t<0$.
%We denote by  the \textit{distribution function} of $u$, which measures the volume of the sub level-set of $-\vert u\vert$.
It is clear that $\mu$
is increasing from $\mu(-\infty)=0$
to $\mu(0)$. For simplicity, we abbreviate the set $\{x\in\S:u(x)<t\}$ simply by $\{u<t\}$. %=\vert{\rm spt}u\vert$. 

The \textit{increasing rearrangement} of $u$ is denoted by $u_\ast:[0,\infty]\ra[-\infty,0]$, and is defined by
\begin{align*}
    u_\ast(s)=\sup\{t\leq0:\mu(t)<s\}.
\end{align*}
The \textit{convex symmetrization of $u$ in $\S$} is given by
\begin{align*}
    (u_\star)_{F,\S}(x):=u_\ast(\kappa_{F,\S}(F^o(x))^n),
\end{align*}
where $\kappa_{F,\S}=|\mcW\cap \S|$. For simplicity, we omit the subscript $(F,\S)$ and denote $$u_\star:= (u_\star)_{F,\S}, \quad \kappa=\kappa_{F,\S}.$$

\begin{remark}
\normalfont
When $F$ is the Euclidean norm, the corresponding relative Schwarz symmetrization in $\S$, which we shall denote by $u_\#$ below, has been considered in \cite{PT86} and \cite{LP90}. On the other hand, When $F$ is a norm and $\S=\mfR^n$, the corresponding  convex symmetrization  has been considered in \cite{AFT97}.
\end{remark}

%------------------------
%\section{Convex rearrangement in a convex cone}

%--------
%\subsection{P\'olya-Szeg\"o inequality}
We first prove the P\'olya-Szeg\"o principle for the convex symmetrization in a convex cone.
\begin{theorem}
    [P\'olya-Szeg\"o principle in a convex cone]\label{Prop-PolyaSzego-ani}
    Let  $p\geq1$ and $u\in W^{1, p}(\S)$ be a non-positive function which vanishes at infinity.
    Then $u_\star$ {is in the same function space as $u$} and the following holds:
    \begin{align}\label{ineq-Polya-Szego-ani}
        \int_{\S}F^p(\na u)\rd x
        \geq\int_{\S}F^p(\na u_\star) \rd x.
    \end{align}
\end{theorem}

\begin{proof}
The proof follows closely that of \cite{AFT97}.

We first assume $u\in C^{\infty}$. By Sard's theorem, $\{u=t\}$ is regular hypersurface for a.e. $t<0$.
The co-area formula gives
\begin{align*}
    \mu(t)=\int_{\{u<t\}}1\rd x
    =\int^t_{-\infty}\left(\int_{\{u=r\}}\frac{1}{\vert\na u\vert}d\mcH^{n-1}\right)\rd r.
\end{align*}
It follows that for a.e. $t<0$, there holds
\begin{align}\label{eq-derivative-mu}
    \frac{\rd}{\rd t}\mu(t)=\int_{\{u=t\}}\frac{1}{\vert\na u\vert}\rd\mcH^{n-1}.
\end{align}
Using the co-area formula again, for a.e. $t<0$, there holds
\begin{align}\label{eq-AFT97-3-1}
    \frac{\rd}{\rd t}\int_{\{u<t\}}F^p(\na u) \rd x
    =\int_{\{u=t\}}\frac{F^p(\na u)}{\vert\na u\vert}\rd\mcH^{n-1}.
\end{align}
Applying the H\"older inequality, we get
\begin{align}\label{ineq-Holder}
    \int_{\{u=t\}}F(\na u)\vert\na u\vert^{-1}\rd\mcH^{n-1}
    \leq\left(\int_{\{u=t\}}\frac{F^p(\na u)}{\vert\na u\vert}\right)^{1/p}
    \left(\int_{\{u=t\}}\frac{1}{\vert\na u\vert}\right)^{1-1/p}
\end{align}
Substituting \eqref{eq-derivative-mu} and \eqref{eq-AFT97-3-1} into \eqref{ineq-Holder}, we obtain
\begin{align}\label{ineq-AFT97-3-2.5}
    \int_{\{u=t\}}\frac{F^p(\na u)}{\vert\na u\vert}\rd\mcH^{n-1}
    \geq\left(\int_{\{u=t\}}F(\frac{\na u}{\vert\na u\vert})\rd\mcH^{n-1}\right)^p\left(\mu'(t)\right)^{1-p}.
\end{align}
Notice that for a.e. $t$, the outward unit normal of $\{u<t\}$ along the boundary $\{u=t\}$ is given by $\nu=\frac{\na u}{\vert\na u\vert}$,
taking \eqref{eq-P_F}, the anisotropic isoperimetric inequality \eqref{ineq-anisotropic-isoperimetric} and also \eqref{eq-AFT97-3-1} into account, integrating \eqref{ineq-AFT97-3-2.5} over $(-\infty,0)$, we arrive at
\begin{align}\label{ineq-AFT97-3-3}
    \int_{\S}F^p(\na u)\rd x
    \geq\int^0_{-\infty}(\mu'(t))^{1-p}\left(n\kappa^{1/n}\mu(t)^{1-1/n}\right)^p\rd t.
\end{align}
It suffice to verify that the RHS of \eqref{ineq-AFT97-3-3} coincides with $\int_{\S}F^p(\na u_\star)\rd x$.

We proceed by noticing that $u_\star$ is anisotropic symmetric, radially increasing,
and hence the sub level-sets of $u_\star$ are homothetic to the unit Wulff sector centered at the origin. This means, the anisotropic isoperimetric inequality holds as an equality for the sets $\{u_\star<t\}$, namely,
\begin{align*}
    \int_{\{u_\star=t\}}F(\frac{\na u_\star}{\vert\na u_\star\vert})\rd\mcH^{n-1}
    =n\kappa^{1/n}\mu(t)^{1-1/n}.
\end{align*}
{On the other hand, the H\"older inequality in \eqref{ineq-Holder} also holds as an equality when $u=u_\star$. This is because $F(\na u_\star)$ is constant along the level-sets $\{u_\star=t\}$, which is due to the fact that $F(\nabla F^o(x))=1$.}

The proof is done by repeating the argument above and noticing that every inequality indeed holds as an equality for $u_\star$. In particular, one gets
\begin{align}\label{eq-AFT97-3-5}
    \int_{\S}F^p(\na u_\star)\rd x
    =&\int^0_{-\infty}\left(\int_{\{u_\star=t\}}\frac{F^p(\na u_\star)}{\vert\na u_\star\vert}\right)\rd t\notag\\
    =&\int^0_{-\infty}(\mu'(t))^{1-p}\left(n\kappa^{1/n}\mu(t)^{1-1/n}\right)^p\rd t.
\end{align}
We complete the proof for $u\in C^\infty$.
The general case $u\in W^{1,p}$ follows from a standard density  argument.
\end{proof}
%If, instead, $u$ is a non-negative function on $\S$, then we see that: for $t<0$, $\mu(t)=\vert\{x\in\S:u(x)>-t\}\vert$ measures the volume of the super level-set of $u$, and the outer unit normal in this case is given by $\nu=-\frac{\nabla u}{\vert \nabla u\vert}$ when it exists.
%Applying the same approach of \cref{Prop-PolyaSzego-ani}, we obtain
%\begin{corollary}\label{Cor-PolyaSzego-ani}
%Let $u=\tilde u\mid_\S$, where $\tilde u\in W^{1,p}_0(\mfR^n)$, $p\geq1$.
%    If $u$ is non-negative,
%    then $u_\star$ {\color{purple}is in the same function space as $u$} and
%   \begin{align}
 %       \int_{\S}F^p(-\na u)dx
 %       \geq\int_{\S}F^p(\na u_\star) dx.
 %   \end{align}
%\end{corollary}
%Denote $$u^+(x)=\max\{u(x),0\},u^-(x)=\max\{-u(x),0\}.$$
%Moreover, if $u\in W_0^{1,p}(\mfR^n)$, then $u^+,u^-\in W_0^{1,p}(\mfR^n)$,
%by combining \cref{Prop-PolyaSzego-ani} with \cref{Cor-PolyaSzego-ani},, we have

%The following follows directly from \cref{Prop-PolyaSzego-ani}. 
%\begin{corollary}
% Let  $p\geq1$ and $u\in W^{1, p}(\S)$ be a function which vanishes at infinity.
 %   Then $u_\star$ {is in the same function space as $u$} and
%    \begin{align*}
%        \int_{\S}F^p\left(-\na (u^++u^-)\right)dx
%        \geq\int_{\S}F^p(\na (-u^+)_\star)+F^p(\na (-u^-)_\star) dx.
%    \end{align*}
%\end{corollary}
 
%\begin{remark}
 %   On the other hand, I think, as for the non-negative function $u$, we can use the classical {\bf decreasing rearrangement}
%($u^\ast$ and $u^\star$) to find that
%\begin{align*}
%    \int_\S F^p(-\na u)dx
%    \geq\int_\S
%    F^p(-\na u^\star)dx.
%\end{align*}}
%\end{remark}}

%--------
\subsection{PDE Comparison Principle}\

Let $\S$ be a open convex cone such that $\p\S\setminus\{0\}$ is smooth. Let $\Om\subset\S$ be a bounded domain such that $\Gamma:=\overline{\p\Om\cap\S}$, the topological closure of $\p\Om\cap\S$ in $\mfR^n$, is a smooth hypersurface with boundary and $\Gamma_1:=\p\O\setminus\Gamma$.
We always assume that $\mcH^{n-1}(\Gamma_1)>0$, and $\mcH^{n-1}(\Gamma)>0$. Such a domain is called a \textit{sector-like} domain.
We use $\nu$ to denote the outward unit normal of $\p\Om$, when it exists.

We consdier the following  mixed boundary value problem for elliptic equations of divergence type in $\O$.
\begin{align}\label{eq-M-ani}
	\begin{cases}
		-{\rm div}\left(a(x,u,\na u)\right) =f\quad&\text{in }\Om,\\
		u=0\quad&\text{on }\Gamma,\\
		-a(x,u,\na u)\cdot \nu=0\quad&\text{on }\Gamma_1,\tag{M}
	\end{cases}
\end{align}
where \begin{align}\label{property-f}
    f\le 0, f\in L^{\frac{2n}{n+2}}\hbox{ if } n\ge 3\hbox{ and } f\in L^p (p>1) \hbox{ if }n=2.
\end{align}
$a(x,\eta,\xi)=\{a_i(x,\eta,\xi)\}_{i=1,\cdots, n}$ are Carath\'eodory functions satisfying:
	\begin{align}\label{condition-elliptic-ani}
		a(x,\eta,\xi)\cdot\xi
		\geq F^2(\xi),\quad\text{for a.e. }x\in\Om,\text{ }\eta\in\mfR,\text{ }\xi\in\mfR^n.
	\end{align}

We write $W_0^{1,2}(\Om;\Gamma)$ to be the space of functions lying in $W^{1,2}(\Om)$ which has vanishing trace on $\Gamma=\overline{\p\Om\cap \S}$. $u\in W_0^{1,2}(\Om;\Gamma)$ is said to be a weak solution of \eqref{eq-M-ani} if it satisfies
\begin{align}\label{eq-weaksol-ani}
	\int_\Om\left(a(x,u,\na u)\cdot \na v\right)\rd x
	=\int_\Om fv\rd x,\quad\forall v\in W^{1,2}_0(\Om;\Gamma).
\end{align}
In the case  $a(x,\eta,\xi)=\frac12F^2(\xi)$, we denote $$\Delta_F u= {\rm div}[\nabla_\xi(\frac12F^2)(\nabla u)].$$

The aim of this subsection is to establish a comparison principle for \eqref{eq-M-ani}. Let $\Om_\star$ be the Wulff sector centered at the origin with the same volume as $\O$ and  $\Gamma_\star=\overline{\p \Om_\star\cap \S}$ and $(\Gamma_1)_\star=\p \S\setminus \Gamma_\star$.
\begin{theorem}\label{Thm-rearrangement-ani}
	Let $u\in W^{1,2}_0(\Om;\Gamma)$ be a solution to \eqref{eq-M-ani}. If $z\in W_0^{1,2}(\Om_\star;\Gamma_\star)$ is the solution of the following mixed boundary value problem
	\begin{align}\label{eq-MBP-z-ani}
		\begin{cases}
			-\De_F z=f_\star\quad&\text{in }\Om_\star,\\
			z=0\quad&\text{on }\Gamma_\star,\\
			\na_\xi(\frac12F^2)(\na z)\cdot\nu=0\quad&\text{on }(\Gamma_1)_\star,
		\end{cases}
	\end{align}
	then 
	\begin{align}\label{ineq-ustar-z-ani}
	0\ge u_\star(x)\geq z(x)\quad\text{for any }x\in\Om_\star.
	\end{align}
	%In other words, among all types of mixed boundary equations in the form of \eqref{eq-M-ani}, the standard type \eqref{eq-MBP-z-ani} gives the smallest solution; among all types of domain $\Om$ in the convex cone $\S$, Wulff sector minimizes the solution to \eqref{eq-M-ani}.
\end{theorem}
\begin{remark}
\normalfont
One sees that if $z$ is radially symmetric with respect to $F$, namely, $z(x)=\bar z(F^o(x))$ for some one-variable function $\bar z$, then $z$ automatically satisfies $\na_\xi(\frac12F^2)(\na z)\cdot\nu=0$ on $(\Gamma_1)_\star$. Hence it follows from the maximum principle that the solution $z$ is radially symmetric with respect to $F$.
\end{remark}
We first see that the solution to \eqref{eq-weaksol-ani} is non-positive.
\begin{lemma}\label{Lem-nonpositive}
	If $u$ is a weak solution of the mixed boundary equation \eqref{eq-weaksol-ani}, then
$u\leq0$  in $\Om$.	In particular $u_\star\le 0$ in $\Om_\star$ and the weak solution of \eqref{eq-MBP-z-ani}  $z\le 0$ in $\Om_\star$.
\end{lemma}
\begin{proof}
%	Notice that $u$ satisfies
%	\begin{align*}
%		\int_{\Om} a(x, u,\na u)\cdot\na\varphi dx=\int_{\Om}f \varphi dx,\quad\text{for all }\varphi\in W_0^{1,2}(\Om;\Gamma).
%	\end{align*}
By testing the definition of weak solution \eqref{eq-weaksol-ani} with $\varphi=u^+=\max\{0,u\}\in W_0^{1,2}(\Om;\Gamma)$, the ellipticity \eqref{condition-elliptic-ani} of $a$ and the non-positive of $f$ imply
\begin{align*}
	0\geq&\int_{\{u>0\}}f u^+\rd x
	\ge\int_{\{u>0\}}F^2(\na u(x))\rd x=\int_{\Om}F^2(\na u^+(x))\rd x.
\end{align*}
It follows that  $u\leq0$ in $\Om$.
\end{proof}
\begin{proof}[Proof of \cref{Thm-rearrangement-ani}]
	%------------------------modified proof---------
	
	We follow closely the classical proof in \cite{Talenti76}.
	
	{\bf Claim 1. }for any $u\in W_0^{1,2}(\Om;\Gamma)$, the following inequality \begin{align}\label{ineq-Talenti76-(40)}
	n^2\kappa^{2/n}\leq\mu(t)^{-2+2/n}\mu'(t)\left(\frac{\rd}{\rd t}\int_{\{u<t\}}F^2(\na u)\rd x\right)
	\end{align}
holds for a.e. $t<0$.

Indeed, by virtue of the fact that $u\in W_0^{1,2}(\Om;\Gamma)$, we know that $u$ is of bounded variation in $\Om$, so that the co-area theorem for BV functions (see e.g., \cite[Theorem 5.9]{EG15}) gives: the sets $\{u<t\}$ have finite perimeter (whose boundary is then given by $\{u=t\}$ and the outer unit normal is $\frac{\na u}{\vert\na u\vert}$) for a.e. $t$.
We can then use the co-area formula and recalling \eqref{eq-P_F} to see that 
\begin{align*}
	\int_{\{u<t\}}F(\na u)\rd x
	=\int_{-\infty}^t\int_{\{u=s\}}\frac{F(\na u)}{\vert\na u\vert}\rd\mcH^{n-1}\rd s
	=\int_{-\infty}^tP_F(\{u<s\};\S)\rd s,
\end{align*}
which implies: for a.e. $t<0$,
\begin{align*}
	\frac{\rd}{\rd t}\int_{\{u<t\}}F(\na u)\rd x
	=P_F(\{u<t\};\S).
\end{align*}
Using the anisotropic isoperimetric inequality \eqref{ineq-anisotropic-isoperimetric}, we find
\begin{align}\label{ineq-Talenti76-(43)}
	\frac{\rd}{\rd t}\int_{\{u<t\}}F(\na u)\rd x
	\geq n\kappa^{1/n}\mu(t)^{1-1/n}.
\end{align}
On the other hand, writing $\frac{\rd}{\rd t}\int_{\{u<t\}}F(\na u)\rd x$ in the form of differential quotients, we obtain
\begin{align}\label{ineq-Talenti76-(44)}
		\frac{\rd}{\rd t}\int_{\{u<t\}}F(\na u)\rd x
		=&\lim_{h\ra0}\frac{\int_{\{t<u<t+h\}}F(\na u)\rd x}{h}\notag
		\\ \leq&\lim_{h\ra0}\frac{\vert\{ t<u<t+h\}\vert^{1/2}}{h^{1/2}}\frac{\left(\int_{\{t<u<t+h\}}F^2(\na u)\rd x\right)^{1/2}}{h^{1/2}}\notag\\
		=&\left(\mu'(t)\right)^{1/2}\left(\frac{\rd}{\rd t}\int_{\{u<t\}}F^2(\na u)\rd x\right)^{1/2}.
\end{align}
\eqref{ineq-Talenti76-(40)} follows from \eqref{ineq-Talenti76-(43)} and \eqref{ineq-Talenti76-(44)}, which proves {\bf Claim 1}.

{\bf Claim 2.} for any weak solution $u$ to \eqref{eq-M-ani}, the function
\begin{align*}
	\Psi(t):=\int_{\{u<t\}}F^2(\na u)\rd x
\end{align*}
is an increasing function on $-\infty<t<0$, with
\begin{align*}
	0\leq\Psi'(t)\leq\int_0^{\mu(t)}-f_\ast(s)\rd s.
\end{align*}

	For $t\leq0$, by testing \eqref{eq-weaksol-ani} with the following truncated function
	\begin{align*}
		v_h:=
		\begin{cases}
			-h\quad&\text{if }u<t-h,\\
			u-t\quad&\text{if }t-h<u< t,\\
			0\quad&\text{if }u\geq t,
		\end{cases}
	\end{align*}
	we find
	\begin{align*}
		\Psi(t)-\Psi(t-h)=&\int_{\{t-h<u<t\}}F^2(\na u)\rd x
		\leq \int_{\{u<t\}}fv_h\rd x\notag
		\\=&\int_{\{t-h<u<t\}}f\cdot(u-t)\rd x-h\int_{\{u<t-h\}}f\rd x,
		%{\color{purple}+\int_{T\cap\{u<t\}}gv_h},
	\end{align*}
	dividing both sides by $h$ and sending $h\searrow0$, by virtue of the integrability of $f$ and $u$, we thus have: for a.e. $t<0$,
	\begin{align*}
		\frac{\rd}{\rd t}\int_{\{u<t\}}F^2(\na u)\rd x
		=\Psi'(t)
		\leq\int_{\{u<t\}}(-f)\rd x.
		%{\color{purple}-\int_{T\cap\{u<t\}}g	\leq\int_{\{u<t\}}fdx}
		%\leq\int_{\{u<t\}}\vert f\vert dx.
	\end{align*}
	The Hardy-Littlewood inequality yields that
 \begin{align*}
		\int_{\{u<t\}}(-f)\rd x\le \int_0^{\mu(t)}-f_\ast(s)\rd d s.
	\end{align*}  {\bf Claim 2} follows.

	A crucial consequence of {\bf Claim 1} and {\bf Claim 2} is the following inequality:
	
	%----------------------------modified proof--------
	\begin{align*}
		1
		\leq\frac{\mu'(t)}{n^2\kappa^{2/n}\mu(t)^{2-2/n}}\int_0^{\mu(t)}-f_\ast(s)\rd s.
	\end{align*}
	%-----------------------------------------------------------------------------------
	Notice that the RHS is the derivative of an increasing function of $t$.
	Integrating both sides over $(t,0)$, one gets
	\begin{align*}%\label{ineq-AFT-4-3}
		t
		\geq\frac{1}{n^2\kappa^{2/n}}\int_{\mu(t)}^{\vert\Om\vert}r^{-2+2/n}\rd r\int_0^rf_\ast(s)\rd s.
	\end{align*}
	Invoking again the definition of the increasing rearrangement, we thus find
	\begin{align}\label{ineq-ustar-v}
		u_\ast(s)
		\geq\frac{1}{n^2\kappa^{2/n}}\int_s^{\vert\Om\vert}r^{-2+2/n}\rd r\int_0^rf_\ast(s)\rd s.
	\end{align}
	%From this and the definition of the capillary Schwarz symmetrization in a half-space, we indeed have
	%\begin{align}
	%    u^\star(x)=u^\ast\left(C_\theta\vert x-r_u\cos\theta E_n\vert^n\right)
	%    \leq\frac{1}{n^2C_\theta^{2/n}}\int_{C_\theta\vert x-r_u\cos\theta E_n\vert^n}^{\vert\Om\vert}r^{-2+2/n}dr\int_0^rf^\ast(s)ds.
	%\end{align}
	%The proof is done by checking that the RHS equals to the solution of \eqref{eq-M-v}.
	
 %It is easy to recognize that the function on the RHS of \eqref{ineq-AFT-4-3} is proportional to the increasing rearrangement of the solution of a specific type of capillary symmetrized problem. 
 By a standard ODE computation, we know that
 \begin{align*} 
		v_\ast(s)=\frac{1}{n^2\om^{2/n}}\int_s^{\vert\Om\vert}r^{-2+2/n}\rd r\int_0^rf_\ast(s)\rd s
	\end{align*}
	where $\om=|B_1\cap\S|$ and $v_\ast(s)$ is the increasing rearrangement of the solution $v$ of the mixed boundary problem:
	\begin{align}\label{eq-MBP-v}
		\begin{cases}
			\De v=-f_\#\quad\text{in }\Om_\#,\\
			v=0\quad\text{on }\Gamma_\#,\\
			\nabla v\cdot\nu=0\quad\text{on }(\Gamma_1)_\#,
		\end{cases}
	\end{align}
	where $\Om_\#=B_r\cap \S$ for some $r$ such that $\vert\Om_\#\vert=\vert\Om\vert$ and $f_\#$ is the Schwarz symmetrization of $f$  (that is, the convex symmetrization when $F$ is the Euclidean norm).

Hence, \eqref{ineq-ustar-v} can be rewritten as
	\begin{align}\label{ineq-AFT-4-4}
		u_\ast(s)
		\geq\frac{\om^{2/n}}{\kappa^{2/n}}v_\ast(s).
	\end{align}
	%For the sake of clarification, we include the proof and delay it to \cref{App-1}.

	{\bf Claim 3. }For $z=z_\star\in W_0^{1,2}(\Om_\star;\Gamma_\star)$ that solves \eqref{eq-MBP-z-ani}, there holds
	\begin{align*}
		\frac{\om^{2/n}}{\kappa^{2/n}}v(x)=z_\#(x),\quad\text{for }x\in\Om_\#.
	\end{align*}
 Consider the functional 
 \begin{align*}
		\mathcal{F}(w)=\int_{\Om_\star}\left(\frac{1}{2}F^2(\na w)-f_\star w\right)\rd x,\quad\text{for }w\in W_0^{1,2}(\Om_\star;\Gamma_\star).
	\end{align*}
 It is clear that $z=z_\star$ is the minimizer for $\mathcal{F}$, which is non-positive by \cref{Lem-nonpositive} and radially symmetric with respect to $F$. Hence
  \begin{align*}
		\mathcal{F}(z)=\mathcal{F}(z_\star)=\int_{\Om_\#}\frac{1}{2}\frac{\kappa^{2/n}}{\om^{2/n}}\vert \na z_\#\vert^2-\int_{\Om_\#}f_\# z_\#.
	\end{align*}
 Note that, by Poly\'a-Szeg\"o for the Euclidean norm, for any $z\in W_0^{1,2}(\Om_\star;\Gamma_\star)$,
  \begin{align*}
		\mathcal{F}_\#(z):=\int_{\Om_\#}\frac{1}{2}\frac{\kappa^{2/n}}{\om^{2/n}}\vert \na z\vert^2-\int_{\Om_\#}f_\# z\ge \mathcal{F}_\#(z_\#).
	\end{align*}
 Hence $z_\#$ minimizes the functional $\mathcal{F}_\#$. It follows that $\frac{\om^{2/n}}{\kappa^{2/n}}z_\#$ solves \eqref{eq-MBP-v}, {\bf Claim 3} follows.

	Finally, {\bf Claim 3} together with \eqref{ineq-AFT-4-4} implies that
	\begin{align*}
		u_\star(x)\geq z_\star(x)=z(x),\quad x\in\Om_\star,
	\end{align*}
	where $z\in W_0^{1,2}(\Om_\star;\Gamma_\star)$ is a solution to \eqref{eq-MBP-z-ani}.
	This completes the proof.
\end{proof}
\section{Capillary gauge in the half-space}\label{Sec-3}

In this section, we first introduce a gauge in $\mfR^n_+$ (as a special case of convex cone), by virtue of which we transform the study of capillary problem in the half-space to the study of related anisotropic problem with respect to such gauge in $\mfR^n_+$.

%let $\Om\subset\mfR^n_+$ be a bounded, smooth, connected, relatively open subset.
%We set $M:=\overline{\p\Om\cap{\rm int}(\mfR^n_+)}$, $T:=\p\Om\setminus M$, $E_n=(0,\ldots,0,1)$.

%\subsection{Capillary {\color{purple}gauge} in a half-space}
Denote $E_n=(0,\ldots,0,1)$. Given $\theta\in(0,\pi)$, let $F_\theta:\mfR^n\ra\mfR^+$ be given by
\begin{align}\label{defn-capillary-norm}
    F_\theta(\xi)=\vert\xi\vert-\cos\theta\left<\xi,E_n\right>.
\end{align}
%For simplicity, we omit the subscript $\theta$ in all follows.
It is direct to see that $F_\theta$ is indeed a gauge and it is smooth on $\mfR^n\setminus\{0\}$. We call it {\it capillary gauge} (see \cref{dualgauge1} for the reason). Note that $F_\theta$ is not even except for the case $\theta=\frac{\pi}{2}$. 
Since $$\na F_\theta(\xi)=\frac{\xi}{|\xi|}-\cos\theta E_n,$$ one sees that the Wulff shape with respect to $F_\theta$ is given by $$\na F_\theta(\mfS^{n-1})=\mfS^{n-1}-\cos\theta E_n=\{|x+\cos\theta E_n|<1\}.$$
\begin{proposition}\label{dualgauge}
 The dual gauge $F_\theta^o: \mfR^n\to\mfR$ is given by
 \begin{align*}
    F_\theta^o(x)=\frac{\vert x\vert^2}{\sqrt{\cos^2\theta\left<x,E_n\right>^2+\sin^2\theta\vert x\vert^2}-\cos\theta\left<x,E_n\right>}.
\end{align*}
\end{proposition}
\begin{proof}

Consider the convex body $K$ determined by $\na F_\theta(\mfS^{n-1})=\{|x+\cos\theta E_n|=1\}$. We shall find the radial function for $K$. Let $y\in \na F_\theta(\mfS^{n-1})$ be given by $y=\rho(x)x, x\in\mfS^{n-1}$. Thus
$$|\rho(x)x+\cos\theta E_n|=1.$$
It follows that 
$$\rho(x)=\sqrt{\cos^2\theta\left<x,E_n\right>^2+\sin^2\theta}-\cos\theta\left<x,E_n\right>.
$$
That is, the radial function for $K$ is given by $\rho: \mfS^{n-1}\to \mfR$ as above. A classical result in the theory of convex bodies says that the support function for the dual convex body $K^o$ is equal to the reciprocal of the radial function of $K$, see e.g., \cite[(1.52)]{Schneider14}. On the other hand, $F^o_\theta$, when restricting on $\mfS^{n-1}$, is exactly the support function for $K^o$.
Therefore, we see that $F^o_\theta: \mfS^{n-1}\to\mfR$ is given by
$$F_\theta^o(x)=\frac{1}{\rho(x)}=\frac{1}{\sqrt{\cos^2\theta\left<x,E_n\right>^2+\sin^2\theta}-\cos\theta\left<x,E_n\right>}.
$$
By one-homogeneous extension of $F_\theta^o$ to $\mfR^n$, that is, $F_\theta^o(x)=|x|F_\theta^o(\frac{x}{|x|})$, we get the assertion.
\end{proof}
\begin{proposition}\label{dualgauge1}
 The Wulff ball $\mcW_{r,\theta}$ of radius $r$ centered at the origin, with respect to $F_\theta$, is given by $B_r(-r\cos\theta E_n),$ the Euclidean ball of radius $r$ centered at $-r\cos\theta E_n$.
 In particular, $\p\mcW_{r,\theta}$ intersects with the hyperplane $\p\mfR_+^n=\{x_n=0\}$ at the contact angle $\theta$.
\end{proposition}
\begin{proof}
Using \cref{dualgauge}, it is direct to check that
$F_\theta^o(x)<r$ is equivalent that $|x+r\cos\theta E_n|<r$, the first assertion follows. For any $z\in\p B_r(-r\cos\theta E_n)\cap\{x_n=0\}$,
\begin{align*}
    \left<\frac{z-(-r\cos\theta E_n)}{r},E_n\right>=\cos\theta.
\end{align*}
The second assertion follows. 
\end{proof}

%we note that $C_\theta$ is exactly defined as the volume of the unit Wulff sector centered at the origin; and also $n$ times the anisotropic relative perimeter of $\mfW$ in $\mfR^n_+$ (and hence the free energy of $\mfW$, see \cref{Lem-free-energy}).

We set
\begin{align*}
    {\bf b}_\theta:=|\mcW_{1,\theta}\cap \mfR_+^n|.%{\color{red}=\left(\frac{\theta\om_n}{\pi}-\frac{\sin^{n-1}\theta\cos\theta\om_{n-1}}{n}\right)}.
\end{align*}
%The second equality follows from a direct computation. 
One sees easily that $$|\mcW_{r,\theta}\cap \mfR_+^n|={\bf b}_\theta r^n.$$

Using the gauge $F_\theta$, we observe that
 the classical free energy functional can be reformulated as anisotropic area functional.  
\begin{proposition}\label{Lem-free-energy}
Let $E$ be a set of finite perimeter in $\mfR^n_+$.
Then $$P_{F_\theta}(E;\mfR^n_+)=P(E;\mfR^n_+)-\cos\theta P(E;\p\mfR^n_+).$$
\end{proposition}

\begin{proof}
Since ${\rm div}(E_n)=0$, using the divergence theorem, one gets
\begin{align*}
    0=\int_\O {\rm div}(E_n)\rd x
    =\int_{\p^\ast E\cap \mfR^n_+}\left<\nu_E,E_n\right>\rd\mcH^{n-1}-P(E;\p\mfR^n_+).
\end{align*}
On the other hand, the definition of $F_\theta$ yields
\begin{align*}
    P_{F_\theta}(E;\mfR^n_+)=\int_{\p^\ast E\cap \mfR^n_+} F_\theta(\nu_E)\rd\mcH^{n-1}
    =P(E;\mfR^n_+)-\cos\theta\int_{\p^\ast E\cap \mfR^n_+}\left<\nu_E,E_n\right>\rd\mcH^{n-1}.
\end{align*}
This completes the proof.
\end{proof}
From this, we see that the classical relative isoperimetric inequality in $\mfR^n_+$, 
$$P(E;\mfR^n_+)-\cos\theta P(E;\p\mfR^n_+)\ge  n {\bf b}_\theta^{\frac{1}{n}} |E|^{\frac{n-1}{n}},$$
is equivalent to the anisotropic isoperimetric inequality with respect to $F_\theta$,
$$P_{F_\theta}(E;\mfR^n_+)\ge n {\bf b}_\theta^{\frac{1}{n}} |E|^{\frac{n-1}{n}},$$
where equality holds if and only if $E=\mcW_{r,\theta}\cap \mfR^n_+$, up to a translation on $\p \mfR^n_+$.

In the same spirit, from \cite[Theorem 2]{CNV04} and \cite[Theorem A.1]{CFR20}, we get the following optimal Sobolev inequality.
\begin{theorem}
     Given $\theta\in(0,\pi)$ and $1<p<n$, let $u\in \dot{W}^{1,p}(\mfR^n_+):=\{u\in L^{\frac{np}{n-p}}(\mfR^n_+):\nabla u\in L^{p}(\mfR^n_+)\}$ be a non-positive function.
    Then
$$\|u\|_{L^{\frac{np}{n-p}}(\mfR^n_+)}\le  C_{\theta, p}\left(\int_{\mfR^n_+} (|\nabla u|-\cos\theta\<\nabla u, E_n\>)^{p}\right)^{\frac{1}{p}}.$$
Here $C_{\theta, p}$ is given by
$$C_{\theta, p}=\frac{1}{(\int_{\mfR^n_+} (|\nabla U_{\theta, p}|-\cos\theta\<\nabla U_{\theta, p}, E_n\>)^{p})^{\frac{1}{p}}},$$
with $$U_{\theta, p}(x)=-\left(\frac{1}{\sigma_{p,\theta}+F_\theta^o(x)^{\frac{p}{p-1}}}\right)^{\frac{n-p}{p}},$$ and $\sigma_{p,\theta}>0$ is determined by $\|U_{\theta, p}\|_{L^{\frac{np}{n-p}}(\mfR^n_+)}=1.$
Equality holds if and only if $$u(x)=CU_{\theta, p}(\lambda (x-x_0))$$ for some constant $C\ge 0, \lambda\neq 0$ and some point $x_0\in \p\mfR^n_+$.

\end{theorem}
\begin{remark}
\normalfont
    It has been stated in \cite[Theorem A.1]{CFR20} that the Sobolev inequality holds for possibly sign-changed $u$. However, since $F_\theta$ here is not even, from the proof, one has to restrict to non-negative or non-positive functions.
\end{remark}

%As another geometric application, we note that by using the capillary gauge \eqref{defn-capillary-norm}, we may obtain the Heintze-Karcher-type inequality (\cite[Theorem 1.1]{JWXZ22-A}) and Alexandrov-type theorem (\cite[Corollary 1.2]{JWXZ22-A}) for capillary hypersurfaces in a half-space by letting $F=F_\theta$ in \cite[Theoorem 1.2, Theorem 1.1]{JWXZ22-B}, the conclusions then follow from \cref{dualgauge1} and the following observation.
%\begin{proposition}
%    For a smooth hypersurface $\S\subset\mfR^{n+1}$, the second fundamental form $h$ agrees with the anisotropic second fundamental form $S_{F_\theta}$.
%\end{proposition}
%\begin{proof}
%    At any $x\in\S$, since
%    \begin{align*}
        %S_{F_\theta}\mid_x=\nabla^2_{\xi\xi}F_\theta\mid_{\nu(x)}\circ h\mid_x,
  %  \end{align*}
  %  we get the assertion from the fact that $\nabla^2_{\xi\xi}F_\theta\mid_{\nu(x)}={\rm Id}$.
%\end{proof}

%----------------------------------------------------------------------------------------------------------
\section{Capillary Schwarz symmetrization in the half-space}\label{Sec-4}

We define the capillary Schwarz symmetrization in a rather direct manner. 
Given a non-positive measurable function $u: \mfR^n_+\to (-\infty, 0]$, which vanishes at infinity, we set $r_t$ to be the radius of $\mcW_{r_t,\theta}=B_{r_t}(-r_t\cos\theta E_n)$ %the $\theta$-ball\footnote{By $\theta$-ball we mean, the part of a ball that lies in the upper half-space $\mfR^n_+$, having constant contact angle $\theta$ with $\{x_n=0\}$} 
such that $${\bf b}_\theta r_t^n=|\mcW_{r_t,\theta}\cap\mfR^n_+|=|\{x\in\mfR^n_+: u(x)<t\}|=\mu(t).$$ 
The capillary symmetrization of $u$ is defined as
\begin{align*}
    u_\star(x):=\sup\{t\leq0: \mu(t)< {\bf b}_\theta\vert x-(-r_t\cos\theta E_n)\vert^n\}=\sup\{t\leq0: r_t< \vert x+r_t\cos\theta E_n\vert\}.
\end{align*}
%or equivalently,
%\begin{align*}
   % u_\star(x):=\inf\{t\leq0:\mu(t)>C_\theta\vert x-(-r_t\cos\theta E_n)\vert^n\}.
%\end{align*}
%Since $u^\#$ is radially symmetric with respect to the center $-r_u\cos\theta E_n$, it follows immediately that
%\begin{align*}
%    \vert x\vert^n=2\vert\{u>t\}\vert/\om_n,\quad\text{for any }x\in\{u^\#=t\},
%\end{align*}
%and hence
%\begin{align*}
%    \vert\{x\in\mfR^n_+:u^\#(x)>t\}\vert=\vert\{u>t\}\vert.
%\end{align*}
By definition, one sees readily that for any $t<0$,  the sub level-set $\{u_\star<t\}$ of the rearranged function $u_\star$
is given by some $\mcW_{r_t,\theta}\cap\mfR^n_+$ that has the same measure with $\{u<t\}$. This agrees with the classical idea for Schwarz symmetrization.

Let us proceed by recalling the capillary gauge $F_\theta$ and its dual $F^o_\theta$. As in the proof of \cref{dualgauge1}, we see that $F_\theta^o(x)>r$ is equivalent that $|x+r\cos\theta E_n|>r.$ Therefore, 
the capillary symmetrization $u_\star$ of $u$ can be reformulated as
\begin{align*}
    u_\star(x)=u_\ast\left({\bf b}_\theta (F^o_\theta(x))^n\right),
\end{align*}
where $u_\ast$ is the increasing arrangement.
In the special case $\theta=\pi/2$, we see $F^o_\theta(x)=\vert x\vert$, and
the capillary symmetrization is just
\begin{align*}
    u_\star(x)=u_\ast(\frac{\om_n\vert x\vert^n}{2}).
\end{align*}
From this point of view, we can translate the result in \cref{Sec-2} to the capillary symmetrization.
%----------------
The following is the corresponding P\'olya-Szeg\"o principle, following \cref{Prop-PolyaSzego-ani}.
\begin{theorem}
     Let  $p\geq1$ and $u\in W^{1, p}(\mfR^n_+)$ be a non-positive function which vanishes at infinity.
    Then $u_\star$ is in the same function space as $u$ and the following holds:
    \begin{align}\label{ineq-Polya-Szego}
        \int_{\mfR^n_+}\left(\vert \na u\vert-\cos\theta \na u\cdot E_n\right)^p\rd x
        \geq\int_{\mfR^n_+}\left(\vert \na u_\star\vert-\cos\theta \na u_\star\cdot E_n\right)^p \rd x.
    \end{align}
\end{theorem}
   
%--------------------------------------------------------
Next we consider the following mixed boundary problem for anisotropic PDE with respect to $F_\theta$ in sector-like domain $\O\subset\mfR^n_+$:
\begin{align}\label{eq-M}
    \begin{cases}
			-\De_{F_\theta} u=f\quad&\text{in }\Om,\\
			u=0\quad&\text{on }\Ga=\overline{\p \Om\cap \mfR^n_+},\\
   \nabla_\xi(\frac12F^2_\theta)(\na u)\cdot E_n=0\quad&\text{on }\Ga_1=\p\Om\setminus\Ga.
		\end{cases}
\end{align}
A weak solution $u\in W_0^{1,2}(\Om;\Ga)$ of \eqref{eq-M} satisfies
\begin{align}\label{eq-weaksol}
	\int_\Om \nabla_\xi(\frac12F_\theta^2)(\nabla u)\cdot \na v \rd x
	=\int_\Om fv\rd x
	%+{\color{purple}\int_Tgv}
	,\quad\forall v\in W^{1,2}_0(\Om;\Ga).
\end{align}
By a direct computation, we see 
\eqref{eq-weaksol} is equivalent that
\begin{align}\label{eq-weaksol1}
	\int_\Om (|\nabla u|-\cos\theta\<\nabla u,E_n\>)\left<\frac{\na u}{|\nabla u|}-\cos\theta E_n, \na v\right> \rd x
	=\int_\Om fv \rd x
	%+{\color{purple}\int_Tgv}
	,\quad\forall v\in W^{1,2}_0(\Om;\Ga).
\end{align}
%We say that a function $u\in W_0^{1,2}(\Om;M)$ \textit{solves the capillary PDE in $\Om$}, if it is a weak solution to the following mixed boundary problem.
%\begin{align*}
%    \begin{cases}
%			-\De_F u=f\quad&\text{in }\Om,\\
%			u=0\quad&\text{on }M,\\
%			\na_\xi F(\na u)\cdot E_n=0\quad&\text{on }T.
%		\end{cases}
%\end{align*}
%In particular, by a direct computation, we have
%\begin{align}
 %   \na_\xi F(\na u)
 %   =&\frac{\na u}{\vert\na u\vert}-\cos\theta E_n,\quad \na_\xi F(\na u)\cdot E_n
 %   =\frac{\na u\cdot E_n}{\vert\na u\vert}-\cos\theta,
 %   \end{align}
%    and also
%    \begin{align}
%    &\De_Fu
 %   ={\rm div}(F(\na u)\na_\xi F(\na u))
 %   ={\rm div}\left((\vert \na u\vert-\cos\theta\na u\cdot E_n)(\frac{\na u}{\vert\na u\vert}-\cos\theta E_n)\right)\notag\\
 %   =&\left(1-\cos\theta\frac{\na u\cdot E_n}{\vert\na u\vert}\right)\De u-\frac{2\cos\theta}{\vert\na u\vert}\na^2u[\na u, E_n]
 %   +\frac{\cos\theta \na u\cdot E_n}{\vert\na u\vert^3}\na^2u[\na u,\na u]+\cos^2\theta\na^2u[E_n,E_n].
%\end{align}

%By using \cref{Prop-isoperi-ineq}, we can follow the proof of \cref{Thm-rearrangement-ani} to obtain,
We have the following comparison result for \eqref{eq-M}, following \cref{Thm-rearrangement-ani}.
\begin{theorem}\label{Thm-rearrangement}
	Let $u\in W^{1,2}_0(\Om;\Ga)$ be a weak solution to \eqref{eq-M}, where $f$ satisfies \eqref{property-f}. Let $\Om_\star$ be some $\mcW_{r,\theta}\cap\mfR^n_+$ that has the same measure with $\Om$ and $z\in W_0^{1,2}(\Om_\star)$ be  the solution of the following rearranged mixed boundary problem
	\begin{align}\label{eq-MBP-z}
		\begin{cases}
			-\De_{F_\theta} z=f_\star\quad&\text{in }\Om_\star,\\
			z=0\quad&\text{on }\Ga_\star,\\
			\na_\xi (\frac12F_\theta^2)(\na z)\cdot E_n=0\quad&\text{on }(\Ga_1)_\star,
		\end{cases}
	\end{align}
	then 
	\begin{align}\label{ineq-ustar-z}
		0\geq u_\star(x)\geq z(x)\quad\text{for any }x\in\Om_\star.
	\end{align}
	%In other words, among all types of mixed boundary equations in the form of \eqref{eq-M}, the capillary PDE \eqref{eq-MBP-z} gives the smallest solution; among all types of domain $\Om$ in a half-space, $\theta$-ball minimizes the solution to \eqref{eq-M}.
\end{theorem}
As a particular case, we are interested in the situation when $f=-n$,
we proceed by the following observation, which gives a very well illustration of our motivation to define the capillary rearrangement.
\begin{proposition}\label{Prop-distance-function}The function
    \begin{align}\label{eq-u}
        u(x)=\frac{F_\theta^o(x)^2-r^2}{2}
    \end{align}
 solves
\begin{align}\label{eq-Mixed-u}
    \begin{cases}
        -\De_{F_\theta}u=-n\quad&\text{in }\mcW_{r,\theta}\cap\mfR^n_+,\\
        u=0\quad&\text{on }\p\mcW_{r,\theta}\cap\mfR^n_+,\\
        \na_\xi (\frac12F_\theta^2)(\na u)\cdot E_n=0\quad&\text{on }\bar \mcW_{r,\theta}\cap\p\mfR^n_+.
    \end{cases}
\end{align}
Moreover, if $u$ is  radially symmetric with respect to  $F_\theta$ and solves \eqref{eq-Mixed-u}, then $u$ must be of the form in \eqref{eq-u}.
\end{proposition}
\begin{proof}
    A direct computation by using \eqref{one-homo}  leads to the assertion.
\end{proof}
%From this observation, one can see that the mixed boundary equation that we dealt with is modelled from the
%square of the $F^o$-distance function, and that explains why we call it a capillary PDE.
As a simple but important application of \cref{Thm-rearrangement}, we have
\begin{corollary}\label{cor-bound}
Let $u\in W_0^{1,2}(\Om;\Ga)$ be a weak solution to 
\begin{align*}
    \begin{cases}
        -\De_{F_\theta}u=-n\quad&\text{in }\Om,\\
        u=0,\quad&\text{on }\Ga,\\
        \na_\xi (\frac12F_\theta^2)(\na u)\cdot E_n=0\quad&\text{on }\Ga_1.
    \end{cases}
\end{align*}
Then, $u$ is bounded in $\Om$ with
$$\| u\|_{L^\infty(\Om)}\leq\frac{1}{2}\left(\frac{\vert\Om\vert}{{\bf b}_\theta}\right)^{\frac{2}{n}}.$$
\end{corollary}
\begin{proof}
    By virtue of \cref{Thm-rearrangement}, we know that
    $0\geq u_\star(x)\geq z(x)$ for any $x\in\Om_\star$, where $z=z_\star$ is the radially symmetric solution (with respect to $F_\theta$) to the mix boundary problem of the rearranged PDE
\begin{align}\label{eq-F-theta}
    \begin{cases}
        -\De_{F_\theta}z=-n\quad&\text{in }\Om_\star,\\
        z=0,\quad&\text{on }\Ga_\star,\\
        \na_\xi (\frac12F_\theta^2)(\na z)\cdot E_n=0\quad&\text{on }(\Ga_1)_\star.
    \end{cases}
\end{align}
Thanks to \cref{Prop-distance-function}, we know that $ z(x)=\frac{F_\theta^o(x)^2-r^2}{2}$, where $r$ is the radius of $\Om_\star$, i.e.,
\begin{align*}
    {\bf b}_\theta r^n=|\Om_\star|=\vert\Om\vert.
\end{align*}
%A direct computation then shows that
Hence \begin{align*}
    \vert u_\star\vert\le \vert z\vert
    \leq\frac{1}{2}\left(\frac{\vert\Om\vert}{{\bf b}_\theta}\right)^{\frac{2}{n}}.
\end{align*}
The proof is thus completed by recalling that $\| u\|_{L^\infty}=\| u_\star\|_{L^\infty}$.
\end{proof}

\printbibliography

\end{document}